\documentclass[12pt]{article}
\usepackage{latexsym,color,amsmath,amsthm,amssymb,amscd,amsfonts}
\usepackage{tikz}
\usetikzlibrary{arrows}
\usepackage{scalefnt}
\usepackage[font=small,labelfont=bf]{caption}
\usepackage{float}
\usepackage{graphicx} 
\usepackage{dsfont}
\usepackage{amsopn}
\usepackage{relsize}

\newtheoremstyle{nonum}{}{}{\itshape}{}{\bfseries}{.}{ }{\thmnote{#3}}

\newtheorem{thm}{Theorem}[section]
\newtheorem*{thm*}{Theorem}

\newtheorem{lem}[thm]{Lemma}

\newtheorem{rem}[thm]{Remark}
\newtheorem{conj}[thm]{Conjecture}

\newtheorem{definition}[thm]{Definition}
\newtheorem*{definition*}{Definition}

\newtheorem*{rems*}{Remarks}

\theoremstyle{nonum}

\newcommand{\R}{\mathbb R}
\newcommand{\RR}{\mathbb R}

\newcommand{\N}{\mathbb N}

\def\Vol{{\rm Vol}}

\newcommand{\iprod}[2]{\langle #1,#2 \rangle} 

\def\calL{{\cal L}}

\def\vol{{\rm Vol}}

\def\conv{{\rm conv}}

\begin{document}
\title {On Godbersen's Conjecture}
\date{}
\author{S. Artstein-Avidan, K. Einhorn, D.Y. Florentin, Y. Ostrover }
\maketitle
\begin{abstract}
We provide a natural generalization of a geometric conjecture of 
F\'{a}ry and R\'{e}dei regarding the volume of the convex hull 
of $K \subset {\mathbb R}^n$, and its negative image $-K$.
We show that it implies Godbersen's conjecture regarding the mixed volumes of the
convex bodies $K$ and $-K$. We then use the same type of reasoning to produce the 
currently best known upper bound for the mixed volumes $V(K[j], -K[n-j])$, which 
is not far from Godbersen's conjectured bound.
To this end we prove a certain functional inequality generalizing 
Colesanti's difference function inequality. 
\end{abstract}

\section{Introduction and results}\label{sec:intro}
In this note we consider convex bodies $K\subset \RR^n$, that is,
 compact convex sets with non-empty interior.
The well known Rogers--Shephard bound for the volume of the so
called ``difference body'', $K-K = \{x-y \, | \, x,y \in K \}$, states
that
 \begin{equation}\label{eq:diff-body-ineq} 
 \vol(K-K) \le \binom{2n}{n}\vol(K).\end{equation}
This inequality was proved by Rogers and Shephard in \cite{RS}, where it was also
shown that
equality is attained only for simplices.
By a simplex we mean the convex hull of $n+1$ affinely independent
points in $\RR^n$. Chakerian simplified their argument in \cite{Chakerian},
and in \cite{RS2} they gave another variant of the proof, which we address in the appendix of this text.

A conjectured strengthening of the difference body inequality was suggested
in 1938 by Godbersen \cite{God} (and independently by Makai Jr. \cite{Makai}).

\begin{conj}\label{conj:god}
For any convex body $K\subset \RR^n$ and any $1\le j\le n-1$,  
\begin{equation}\label{eq:Godbersen-conj} V(K[j], -K[n-j])\le \binom{n}{j} \vol(K),\end{equation}
with equality attained only for simplices. 
\end{conj}
 
Here $V(K_1, \ldots, K_n)$ denotes the mixed volume of the $n$ convex bodies $K_1, \ldots, K_n$, and $V(K[j], T[n-j])$ denotes the mixed volume of $j$ copies of the convex body $K$ and $n-j$ copies of the convex body $T$. 
We recall that for convex bodies $K_1, \ldots,K_m \subset {\mathbb R}^n$, and non-negative real numbers $\lambda_1, \ldots,\lambda_m$, a classical result of Minkowski states that the volume of $\sum \lambda_i K_i$ is a homogeneous polynomial of degree $n$ in $\lambda_i$,
\begin{equation}\label{Eq_Vol-Def}
\Vol \left(\sum_{i=1}^m \lambda_i K_i\right) = 
\sum_{i_1,\dots,i_n=1}^m \lambda_{i_1}\cdots\lambda_{i_n} V(K_{i_1},\dots,K_{i_n}),
\end{equation}
and the coefficient $V(K_{i_1},\dots,K_{i_n})$, which depends solely on $K_{i_1}, \ldots, K_{i_n}$, is called the mixed volume of $K_{i_1}, \ldots, K_{i_n}$.
The mixed volume is a non-negative,
translation invariant function, monotone with respect to set inclusion,
invariant under permutations of its arguments, 
and positively homogeneous in each argument. Moreover, one has $V(K[n]) = {\rm Vol}(K)$.
For further information on mixed volumes and their properties, see Section \textsection 5.1 of  \cite{Schneider-book}. 
 
The cases $j=1$ and $j=n-1$ of Conjecture \ref{conj:god} follow from the fact that $-K\subset nK$ for bodies with center of mass at the origin (see \cite{BonFenchel}, page 57). The same argument gives the bound 
\[ V(K[j], -K[n-j])\le n^{\min\{j,n-j\}}\vol(K), \]
for  $0 \leq j \leq n$.
The only other cases for which Conjecture \ref{conj:god} is verified, are simplices (which are the equality case) and bodies of constant width, as shown in \cite{God}. Godbersen's conjecture is indeed a strengthening of the difference body inequality \eqref{eq:diff-body-ineq} since, if Conjecture  \ref{conj:god} holds true, one may write  
\[ \vol(K-K)= \sum_{j=0}^n \binom{n}{j}V(K[j], -K[n-j])\le \sum_{j=0}^n \binom{n}{j}^2\vol(K)= \binom{2n}{n}\vol(K). \]
 
In 1950, F\'{a}ry and R\'{e}dei \cite{FaryRedei}
conjectured that for all convex bodies $K$ of fixed volume, one has
\begin{equation} \label{FR-Conj}  \min_{x\in K}\vol((K-x) \vee (x-K)) \leq \min_{x\in S}\vol((S-x) \vee (x-S)),\end{equation}
where $S$ is a simplex with $\vol(S)=\vol(K)$.
Here, $A\vee B$ denotes the convex hull of the sets $A$ and $B$.
Moreover, they showed that the right-hand side of \eqref{FR-Conj} is precisely $\binom{n}{[n/2]}\vol(S)$.
We conjecture the following generalization of~\eqref{FR-Conj}.

\begin{conj}\label{Conj_Gen-Far-Red}
For any convex body $K\subset \RR^n$ and every $\lambda\in [0,1]$, there exists $x\in K$ such that 
\begin{equation} \label{GFR-Conj}  \vol((1-\lambda)(K-x)\vee \lambda(x-K)) \le  \vol((1-\lambda) S\vee -\lambda S),\end{equation}
where $S$ is a centered simplex, and  $\vol(S) = \vol(K)$. 
\end{conj}

Note that the case $\lambda=1/2$ implies the above mentioned conjecture by F\'{a}ry and R\'{e}dei (since for the simplex, the convex hull of minimal volume is attained when it is centered, see \cite{Einhorn}).
The numerical value of the right-hand side of~$(\ref{GFR-Conj})$ can be computed explicitly (see Section \ref{sec:OCIG}), so that Conjecture 
\ref{Conj_Gen-Far-Red} would give a  numerical upper bound for the quantity on the left-hand side of~$(\ref{GFR-Conj})$. 
Moreover, we remark that in dimension $n=2$, Conjecture 
\ref{Conj_Gen-Far-Red}  holds true (see Section \ref{sec:R2} for a discussion of the planar case).  

Our first result in this paper states that Conjecture  \ref{Conj_Gen-Far-Red} implies Godbersen's conjecture. 

\begin{thm} \label{G.F.R-implies-G}
	Conjecture \ref{Conj_Gen-Far-Red} implies Conjecture \ref{conj:god}.
\end{thm}
 
Our second result is, to the best of our knowledge, the smallest upper bound for $V(K[j], -K[n-j])$ currently known for $2<j<n-2$.
\begin{thm}\label{thm:bound-for-God}
	Let $K\subset \RR^n$ be a convex body, and $1\le j\le n-1$. Then 
	\[ V(K[j], -K[n-j])\le \frac{n^n}{j^j(n-j)^{n-j}}\vol(K)\simeq \binom{n}{j}\sqrt{2\pi \, \frac{j(n-j)}{n}}\vol(K). \] 
\end{thm} 

The proof requires some preparation. 
 Rogers and Shephard  showed in \cite{RS2} that if $0\in K$, then
\begin{equation} \label{RS-ineq-conv-hull} \vol(K \vee  -K )\le 2^n \vol(K), \end{equation}
and that the bound is attained only when $K$ is a simplex with $0$ as a vertex. Another proof for this bound was given by Colesanti in \cite{Colesanti} (see  Section \ref{sec:Colesanti} below). Colesanti's proof is based on a functional analogue of the difference body, which he calls the ``difference function''.
Using monotonicity of mixed volumes,  $(\ref{RS-ineq-conv-hull})$ implies that for every $1\leq j \leq n-1 $, 
\[ V(K[j], -K[n-j]) \le \vol(K \vee -K)  \le 2^n \vol(K).\]

Our next result  generalizes inequality~$(\ref{RS-ineq-conv-hull})$.
\begin{thm}\label{thm:ab-vol}
For any convex body $K\subset \RR^n$ containing the origin and every $\lambda\in[0,1]$
\begin{equation} \label{ineq-needed-for-God} \vol((1-\lambda)K \vee -\lambda K) \le \vol(K).\end{equation}
\end{thm}

Theorem \ref{thm:bound-for-God} is an immediate corollary of Theorem~\ref{thm:ab-vol}. 

\begin{proof}[{\bf Proof of Theorem \ref{thm:bound-for-God}:}] 
Let $1\le j\le n-1$, and set $\lambda = (n-j)/n$. Assume without loss of generality that $0\in K$. 
Since $(1-\lambda) K$ and $-\lambda K$ are contained in $(1-\lambda) K\vee -\lambda K$,  the monotonicity and homogeneity  properties of the mixed volume imply:
\begin{eqnarray*} V(K[j], -K[n-j]) & = &    \frac{1}{(1-\lambda)^j\lambda^{n-j}}V((1-\lambda)  K[j], -\lambda K[n-j])    \\
& \leq &  \frac{1}{(1-\lambda)^j\lambda^{n-j}}\vol((1-\lambda)  K\vee -\lambda K)
\le  \frac{1}{(1-\lambda)^j\lambda^{n-j}}\vol(K),
\end{eqnarray*} 
where the last inequality follows from Theorem~\ref{thm:ab-vol}. Plugging in our choice of $\lambda$ yields precisely the desired bound of the theorem.
\end{proof}

 Theorem~\ref{thm:ab-vol} follows as a special case (where $K= (1-\lambda)K'$, $L=\lambda K'$, and $\theta = \lambda$) from  the following theorem, which is a variation of a result by Rogers and Shephard \cite{RS2}.

\begin{thm}\label{thm:K-L-vol}
Let $K,L\subseteq \RR^n$ be two convex bodies such that $0\in K\cap L$.
For every $\theta\in[0,1],$   
\[ \vol(L \vee -K) \, \vol(\theta K\cap(1-\theta) L) \le  \vol(K)\vol(L). \]		
\end{thm}

We will prove in the appendix that equality in Theorem
\ref{thm:K-L-vol} holds  if and only if $K$ and $L$ are simplices with
a common vertex at the origin and  such that $(1-\theta)L = \theta K$.
Likewise, in Theorem \ref{thm:ab-vol}, equality holds if and only if $K$
is a simplex with a vertex at the origin. 

We offer two  different proofs of Theorem~\ref{thm:K-L-vol}. The first follows from a functional inequality which we turn now to describe. A second proof, which closely follows  Rogers and Shephard's original  argument from  \cite{RS2} and is more geometric in nature, is presented  in the appendix of this paper for completeness.
We start with the notion of a ``$\lambda$-difference
	function".
\begin{definition} \label{def:lambda-col-2-Func}
	Let $\lambda\in(0,1)$, and  $f,g \colon \RR^n\to\RR^+$. The $\lambda$-difference
	function $\Delta_\lambda^{f,g} :\RR^n\to\RR^+$ associated with $f$ and $g$ is defined by
	\[
	\Delta_\lambda^{f,g}(z) = \sup_{(1-\lambda)x+\lambda y = z}f^{1-\lambda}\left(\tfrac {x} {1-\lambda} \right)g^\lambda\left({\tfrac {-y} {\lambda}}\right).
	\]
\end{definition}  

\begin{thm}\label{thm:lambda-col-2-Func}
	Let $f,g:\RR^n\to\RR^+$ be log-concave functions, and $\lambda\in(0,1)$. Then
	\begin{equation} \label{lambda-difference-function-theorem}
  \int_{\RR^n} \Delta_\lambda^{f,g}   \int_{\RR^n} f^{\lambda} g^{1-\lambda} \leq  \int_{\RR^n} f  \int_{\RR^n} g.
	\end{equation}
\end{thm}

The proof of Theorem~\ref{thm:lambda-col-2-Func} is similar to
Colesanti's proof of Theorem 1.1 in \cite{Colesanti}. An immediate
corollary of Theorem~\ref{thm:lambda-col-2-Func}  is the following
result. Let $K^{\circ}$ stand for the polar body of $K$ (see definition
below).

\begin{thm} \label{thm:strange} Let  $K,L \subset {\mathbb R}^n$ be convex bodies. Then 
$$
\vol (K \vee - L ) \,
\vol ((K^\circ +  L^\circ)^\circ ) \le
\vol(K) \, \vol(L).$$
\end{thm}
Note that Theorem \ref{thm:strange} it strictly stronger than Theorem
\ref{thm:K-L-vol}. Indeed, one can readily check that for every $\theta
\in[0,1]$, one has $\theta K \cap (1-\theta) L \subseteq
(K^\circ+L^\circ)^\circ$.

\noindent {\bf Notations:}
A convex body with center of mass at the origin is said to be
\textit{centered}. Given a convex body $K \subseteq \RR^n$, we denote
by $h_K \colon {\mathbb R}^n \to {\mathbb R}$  its \textit {support
function}, that is, the 1-homogeneous convex function given by 
$ h_K(u) = \sup \{ \langle x, u \rangle \ | \  x \in  K \}$.
The \textit{polar body} of $K$ is defined as $K^\circ = \{y\in\RR^n\; |\; \iprod{x}{y}\le 1, \forall x\in K\}$, and is also a convex body.  The \textit{convex indicator function $\mathds{1}_K^\infty$ of $K$} is defined to be zero for $x \in K$ and $+ \infty$ otherwise. 
The \textit{Legendre transform} of   $\varphi \colon \RR^n\to\RR$  is given by $\calL \varphi(x)=\sup_{y\in\RR^n}(\iprod{x}{y}-\varphi(x))$. 
Finally, the \textit{inf convolution} of two functions
$f,g \colon \RR^n \to \RR$ is defined by $(f \square g)(z)=\inf_{x+y=z}
\left\{f(x)+g(y)\right\}$. 

\noindent{\bf Organization of the paper:} 
In Section \ref{sec:OCIG} we prove Theorem~\ref{G.F.R-implies-G}.
In Section \ref{sec:Colesanti} we prove Theorem
\ref{thm:lambda-col-2-Func} and its consequence, Theorem \ref{thm:strange}.
In Section \ref{sec:R2} we discuss the planar case of Conjecture
\ref{Conj_Gen-Far-Red}. Finally, in the appendix, we give for completeness
another proof of Theorem~\ref{thm:K-L-vol}, which is more geometric in
nature, and follows Rogers and Shephard's arguments from \cite{RS2}.

\noindent {\bf Acknowledgment:} The first and second named authors were
partially supported by ISF grant No. 247/11. The third named author was
partially supported by European Research Council grant Dimension 305629.
The fourth named author was partially supported by Reintegration Grant
SSGHD-268274, and by ISF grant No. 1057/10.

\section{Proof of Theorem~\ref{G.F.R-implies-G} }\label{sec:OCIG}

In this section we prove Theorem~\ref{G.F.R-implies-G}. 
We start with the following  lemma regarding the volume of the convex hull of homothetic copies of a simplex $S$ and of $-S$, which is the expression appearing in Conjecture \ref{Conj_Gen-Far-Red}. 

\begin{lem}\label{lem:t-simplex-vol}
	Let $S\subset \RR^n$ be a centered simplex. For $\lambda\in (0,1)$, let $k\in\N$ such that $(n+1)(1-\lambda) -1  \le k\le (n+1)(1-\lambda)$. Then   
	\begin{equation} \label{formula-conv-simplices}  \frac{\vol((1-\lambda)S\vee(-\lambda S))}{\vol(S)} =  \binom{n}{k}(1-\lambda)^k\lambda^{n-k}.\end{equation}
\end{lem}

\begin{rem}{\rm
Usually, the above inequality determines $k=k(\lambda)$ uniquely. When it does not (i.e. there are two such $k$'s), the corresponding expressions on the right hand side of \eqref{formula-conv-simplices} coincide.
}
\end{rem}

\begin{proof}[{\bf Proof of Lemma \ref{lem:t-simplex-vol}}]
Note first that by symmetry, it is enough to assume that $\lambda \leq {\frac 1 2}$. Moreover, since $S$ is centered, for $\lambda \leq {\frac 1 {n+1}}$ one has $-\lambda S \subseteq (1-\lambda)S$  (see \cite{BonFenchel}, page 57), and hence~$(\ref{formula-conv-simplices})$  holds trivially. Thus, we can further assume that ${\frac 1 {n+1}} \leq \lambda$. 
Next, since
\begin{equation}\label{eq:t-and-lambda}
\vol((1-\lambda)S\vee(-\lambda S))=(1-\lambda)^n\vol(S\vee(-\tfrac{\lambda}{1-\lambda} S)),
\end{equation} it is enough to compute the quantity
$\frac{\vol(S\vee(-tS))}{\vol(S)}$ for $t=\frac{ \lambda}{ 1-\lambda}$.
In these notations, our assumptions become
$t \in [ \tfrac {1}{n}, 1 ]$ and $\frac{n+1}{1+t} -1 \le k\le \frac{n+1}{1+t}$.
Moreover, for simplicity, we may take $S$ to be the convex hull of $\{e_j\}_{j=1}^{n+1}$ (vectors of the standard basis of $\RR^{n+1}$) embedded in $\RR^{n+1}$, which has $a=(\frac{1}{n+1},\dots,\frac{1}{n+1})$ as its center of mass. We then wish to compute the $n$-dimensional volume of $K_t := S\vee S_t$, with $S_t$ being the convex hull of the vectors $v_j = (1+t)a-te_j$, where $ j=1,\dots,n+1$.

First, we study the facets of $K_t$. Consider $F_k={\rm conv}\{e_1,\dots ,e_k,v_{k+1},\dots,v_n\}$, for some $k\in\{1,\dots n\}$.  Note that 
$F_k$ lies in the intersection of the two hyperplanes $a+a^\perp$ and $e_1+(u_k)^\perp$, where $$u_k=\Bigl(\underset{k\text{ times}}{\underbrace{-t,\dots,-t}},\underset{n-k\text{ times}}{\underbrace{1,\dots,1}},\left(1+t\right)k-n\Bigr).$$ For $F_k$ to be a facet of $K_t$, we require that for all $j=1,\dots, n+1$, one has  $\iprod{e_j}{u_k} \ge -t$, $\iprod{v_j}{u_k} \ge -t$, and that $\langle u_k, \nu \rangle = -t$ only for $\nu$ which is a vertex of $F_k$. A direct computation shows that this holds if and only if 
 $\frac{n+1}{1+t} -1 < k < \frac{n+1}{1+t}$. Note that for all $j$, $e_j$ and $v_j$ never participate in the same facet, and hence every facet of    $K_t$ consists of either $n$ or $n+1$ vertices. By Remark \ref{rem:only-n-facets} below, we may assume that every facet of $K_t$ is of the form of $F_k$, i.e., consists of exactly $n$ vertices, $k$ of which are vertices of $S$.

The contribution of $F_k\vee \{a\}$ to the ratio $\frac{\vol(K_t)}{\vol(S)}$ can be easily computed by noticing that, by invariance under translation, $\vol_{n}(F_k\vee \{a\})=\vol_{n}(V)$, where $$V={\rm conv}\{0,e_1-a,\dots,e_k-a,-t(e_{k+1}-a),\dots,-t(e_{n}-a)\}.$$ Moreover, since
\begin{eqnarray*}
	\frac{\vol_{n}(V)|a|}{n+1} & = & \vol_{n+1}({\rm conv}\{0,e_1-a,\dots,e_k-a,-t(e_{k+1}-a),\dots,-t(e_{n}-a),a\})\\
	& = & t^{n-k}\vol_{n+1}({\rm conv}\{0,e_1-a,\dots,e_k-a,e_{k+1}-a,\dots,e_{n}-a,a\}),
\end{eqnarray*}
and $\frac{1}{n+1}\vol_{n}(S)|a| = \vol_{n+1}({\rm conv}\{0,e_1,\dots,e_{n+1}\})$, one has $$\frac{\vol_{n}(F_k\vee \{a\})}{\vol_{n} S}=t^{n-k}\det (e_1-a,\dots,e_{n}-a,a)=\frac{t^{n-k}}{n+1}.$$ Finally, summing over all of the facets of $K_t$, we get
$$\frac{\vol(K_t)}{\vol(S)}=(n+1)\binom{n}{k}t^{n-k}\frac{1}{n+1}=\binom{n}{k}t^{n-k},$$ 
which by the definition of $t$ and \eqref{eq:t-and-lambda}, gives the desired result.
\end{proof}

\begin{rem}\label{rem:only-n-facets}
{\rm Unless  $t =\frac{n+1-j}{j}$  (equivalently,  $\lambda=\frac{n+1-j}{n+1}$), for some $j\in \{ [\frac{n}{2}], \dots, n\}$, the value of $k$  (which is now uniquely determined) satisfies the strict inequality $\frac{n+1}{1+t} -1 < k < \frac{n+1}{1+t}$. Thus, in this case, for $F_k$ to be a facet of $K_t$ it must be the convex hull of $n$ vertices. From continuity of the resulting volume ratio formula, it suffices to handle this case only, as done in the proof of Lemma \ref{lem:t-simplex-vol}. 
}
\end{rem}

We are now in a position to prove Theorem~\ref{G.F.R-implies-G}.	
	
\begin{proof}[{\bf Proof of Theorem~\ref{G.F.R-implies-G}}] 
Let $K \subset {\mathbb R}^n$ be a convex body, and assume that
Conjecture \ref{Conj_Gen-Far-Red} holds. We wish to show that Conjecture
\ref{conj:god} holds as well. Note that our assumption is that for every
$\lambda \in [0,1]$ there exists $x\in K$ such that
$$\frac{\vol((1-\lambda)(K-x)\vee \lambda(x-K))}{\vol (K)} \le
  \frac{\vol((1-\lambda)S\vee -\lambda S)}{\vol (S)}           ,$$
where $S$ is a centered simplex. Using the same argument as in the proof
of Theorem \ref{thm:bound-for-God}, the monotonicity, homogeneity in each
argument, and translation invariance of the mixed volume yield, for every
$\lambda\in (0,1)$:
$$\frac{V(K[j],-K[n-j])}{\vol (K)}\le
  \frac{\vol((1-\lambda)(K-x)\vee \lambda (x-K))}
       { (1-\lambda)^j \lambda^{n-j} \vol (K)   }.$$
We may then use Lemma \ref{lem:t-simplex-vol} with $k=j$ and $\lambda
=\frac{n+1-j}{n+1}\in (0,1)$, together with the last two inequalities,
to get 
	$$\frac{V(K[j],-K[n-j])}{\vol (K)}\le \frac
	{\vol((1-\lambda)S\vee -\lambda S)}
	{(1-\lambda)^j \lambda^{n-j} \vol (S)} = \binom{n}{j}.$$
This completes the proof of Theorem~\ref{G.F.R-implies-G}.
\end{proof}

\begin{rem} {\rm
Note that in the above proof we apply Conjecture \ref{Conj_Gen-Far-Red} only for finitely many values of the parameter $\lambda$.  }
\end{rem}

\section{The $\lambda$-difference function}\label{sec:Colesanti} 

In \cite{Colesanti}, Colesanti introduced the so called ``difference function", which is a functional analogue of the difference body notion. He then proved a functional version of the difference body inequality \eqref{eq:diff-body-ineq}
and used it to provide yet another proof of~$(\ref{RS-ineq-conv-hull})$.
In this section, we generalize Colesanti's definition and, using similar methods, extend the main result of \cite{Colesanti}.

First, we recall our notion of a $\lambda$-difference function
(Definition~\ref{def:lambda-col-2-Func}). For two functions
$f,g:\RR^n\to\RR_+$, and $\lambda\in(0,1)$, the $\lambda$-difference
function $\Delta_\lambda^{f,g}$ of $f$ and $g$ is \[
	\Delta_\lambda^{f,g}(z) = \sup_{(1-\lambda)x+\lambda y = z}f^{1-\lambda}\left(\tfrac {x} {1-\lambda} \right)g^\lambda\left({\tfrac {-y} {\lambda}}\right).
	\]
Alternatively if $f=e^{-\varphi},\, g=e^{-\psi}$ we may write 
$$ \Delta_\lambda^{f,g}= e^{-\delta_\lambda^{\varphi,\psi}}, \ {\rm  where} \ \
 \delta_\lambda^{\varphi,\psi}(z)=\inf_{(1-\lambda)x+\lambda y=z}
\left\{(1-\lambda)\varphi(\tfrac {x} {1-\lambda})+\lambda\psi(\tfrac {-y} {\lambda}) \right\}.
$$

\begin{rem} {\rm 

\item 1) The $\lambda$-difference function is compatible
with translations, and multiplications by positive constants. More
precisely, denoting $f_a(x):=f(x+a)$, one has $\Delta_\lambda^{f_a,g_b}=
(\Delta_\lambda^{f,g})_{(1-\lambda)a+\lambda b}\,$ and
$\Delta_\lambda^{af,bg}= a^{1-\lambda}b^\lambda \Delta_\lambda^{f,g}$.

\item 2) Letting $\varPhi(x)=\varphi(x/(1-\lambda))$,
$\Psi(x)=\psi(-x/\lambda)$, we may apply the Pr\'{e}kopa-Leindler
inequality for the functions $e^{-\varPhi},e^{-\Psi}$, and
$e^{-{\delta_\lambda^{\varphi,\psi}}}$, and obtain the following estimate:
$$ \int_{\R^n} \Delta_\lambda^{f,g}  \geq \left((1-\lambda)^{1-\lambda}\lambda^\lambda\right)^n \left(\int_{\R^n}f\right)^{1-\lambda} \left(\int_{\R^n}g \right)^{\lambda}.$$
Theorem~\ref{thm:lambda-col-2-Func}, which we now turn to prove, gives
a complementary upper bound.
}
\end{rem}

\begin{proof}[{\bf Proof of Theorem \ref{thm:lambda-col-2-Func}:}] 
Assume that $f=e^{-\varphi},\,g=e^{-\psi}$, for some convex functions
$\varphi,\,\psi$. Let $z\in\RR^n$. First, assume that there exist
$x^*,y^*\in\RR^n$ such that $$(1-\lambda)x^*+\lambda y^*=z, \ {\rm and} \  \
\delta_\lambda^{\varphi,\psi}(z)=
(1-\lambda)\varphi(x^*/(1-\lambda))+
\lambda \psi(-y^*/\lambda).$$ Using the convexity of $\varphi$ and $\psi$, for every $y\in\RR^n$ one has
	\[
	\psi\left((1-\lambda)y-z/\lambda\right) \le (1-\lambda)\psi(y-x^*/\lambda))+\lambda\psi(-y^*/\lambda),
	\]
	and 
	\[
	\varphi\left(\lambda y\right) \le (1-\lambda)\varphi(x^*/(1-\lambda))+\lambda\varphi(y-x^*/\lambda).
	\]
Summing these two inequalities, we obtain
	\begin{equation} \label{eq-twice-convexity} 
	\psi\left((1-\lambda)y-z/\lambda\right)+\varphi\left(\lambda y\right)
	\le \delta_\lambda^{\varphi,\psi}(z) +
	\lambda\varphi(y-x^*/\lambda) +
	(1-\lambda)\psi(y-x^*/\lambda).
	\end{equation}
As this holds for every $y\in\RR^n$, we integrate~$(\ref{eq-twice-convexity})$ over $y\in\RR^n$ and obtain
	\begin{equation}\label{eq:z2}
	\Delta_\lambda^{f,g}(z)\int_{\RR^n}f^\lambda g^{1-\lambda}\le\int_{\RR^n}g\left((1-\lambda)y-z/\lambda\right)f\left(\lambda y\right)dy.
	\end{equation}
If the case where the infimum in the definition of $\delta_\lambda^{\varphi,\psi}(z)$ is not
attained, there are sequences $(x_j)_{j=1}^\infty, (y_j)_{j=1}^\infty
\subset\R^n$ such that for every $j$, $(1-\lambda) x_j+\lambda y_j=z$,
and $\delta_\lambda^{\varphi,\psi}(z)=\lim_{j\to\infty}
((1-\lambda)\varphi(x_j/(1-\lambda)) + \lambda\varphi(-y_j/\lambda))$.
Using the same argument as before, one has, $$f^{1-\lambda}(x_j/(1-\lambda))
g^\lambda(-y_j/\lambda)  \int_{\RR^n}f^\lambda g^{1-\lambda}\le  
\int_{\RR^n}g\left((1-\lambda)y-z/\lambda\right)f\left(\lambda y\right)dy,$$
and by taking the limit we get that \eqref{eq:z2} holds in this case as
well. We may therefore integrate this inequality with respect to  $z \in \RR^n$, and get
\begin{align*}
	\int_{\R^n}f^\lambda g^{1-\lambda} \int_{\R^n}\Delta_\lambda^{f,g} &\le \int_{\R^n}\left(\int_{\R^n}g\left((1-\lambda)y-z/\lambda\right)f\left(\lambda y\right)dy\right)dz \\
	&=\int_{\R^n}f\left(\lambda y\right)\left(\int_{\RR^n}g\left((1-\lambda)y-z/\lambda\right)dz\right)dy	\\
	&=\int_{\R^n}f\left(\lambda y\right)dy\int_{\RR^n}g\left(-z/\lambda\right)dz \\
	&=\left(\int_{\RR^n}f\right)\left(\int_{\RR^n}g\right).
\end{align*}
This completes the proof of Theorem \ref{thm:lambda-col-2-Func}.
\end{proof}

\begin{rem} {\rm
For $\lambda=1/2$ and $f=g$, Colesanti showed in \cite{Colesanti} that the bound~$(\ref{lambda-difference-function-theorem})$ is sharp and attained for the function
\[
g(x) = 
\begin{cases}
e^{-(x_1+\dots+x_n)}, & \text{if } x_j \ge 0, \quad \forall j=1,\dots, n\\
0,       & \text{otherwise}.
\end{cases}
\]
In fact, this function shows that the bound is sharp for every $\lambda \in (0,1)$. Indeed, one can easily verify that $\Delta_\lambda^{g,g}(z)=e^{-\sum_{i=1}^n|z_j|_\lambda}$, where
\[
|a|_\lambda =  
\begin{cases}
\frac{|a|}{1-\lambda}, & \text{if } a \ge 0 \\
\frac{|a|}{\lambda},       & \text{otherwise}.
\end{cases}
\]
Thus, a direct computation gives 
\[	
	\int_{\RR^n}\Delta_\lambda^{g,g}(z)dz=\int_{\RR^n}e^{-\sum|z_j|_\lambda}dz =
	\sum_{j=0}^n\binom{n}{j}\lambda^j(1-\lambda)^{n-j}\left(\int_{0}^\infty e^{-x}dx\right)^n=1=\int_{\RR^n}g.
\] }
\end{rem}

Theorem~\ref{thm:lambda-col-2-Func} provides, as an immediate corollary, a proof of Theorem \ref{thm:strange}.

\begin{proof}[{\bf Proof of Theorem \ref{thm:strange}}]
Let $K, L \subseteq \RR^n$ be two convex bodies, and without loss of
generality assume they contain the origin. Let $\lambda\in(0,1)$.
Using the homogeneity of the support function $h_{K^{\circ}}(x)$
(both in $K$ and in $x$), one has \[
	\delta_\lambda^{h_{K^\circ},h_{L^\circ}}(z)=\inf_{(1-\lambda)x+\lambda y=z}((1-\lambda)h_{K^\circ}(x/(1-\lambda))+\lambda h_{L^\circ}(-y/\lambda))=
	h_{\frac{1}{1-\lambda}K^\circ}\Box h_{-\frac{1}{\lambda}L^\circ}(z),
\]
where $f\Box g$ is the infimal convolution of $f$ and $g$. It is well known (see e.g.,~\cite{Schneider-book}, Section 1.7), that for lower semi continuous convex functions $f,g$ we have $f\Box g=\calL(\calL f+\calL g)$ and that $\calL h_K = \textbf{1}_K^\infty$. Thus, 
\begin{equation*} \label{eq:delta-h_K}
\begin{split}
	\delta_\lambda^{h_{K^\circ},h_{L^\circ}} & = \calL(\textbf{1}_{\frac{1}{1-\lambda}K^\circ}^\infty  +\textbf{1}_{-\frac{1}{\lambda}L^\circ}^\infty) = \calL(\textbf{1}_{\frac{1}{1-\lambda}K^\circ\cap-\frac{1}{\lambda}L^\circ}^\infty)  \\
& = h_{\frac{1}{1-\lambda}K^\circ\cap-\frac{1}{\lambda}L^\circ}
	=h_{((1-\lambda)K\vee-\lambda L)^\circ}.
\end{split}
\end{equation*}
Moreover, note that $(e^{-h_{K^\circ}})^\lambda
(e^{-h_{L^\circ}})^{1-\lambda}=e^{-h_{\lambda K^\circ+(1-\lambda)L^\circ}}$.
Thus, combining the fact that $\int_{\R^n}e^{-h_{K^\circ}}=n! \vol(K)$ with Theorem \ref{thm:lambda-col-2-Func} for  $f=e^{-h_{K^\circ}}$ and $g =e^{-h_{L^\circ}},$ yields
 $$
 \vol((1-\lambda)K \vee -\lambda L ) \, \vol ((\lambda K^\circ + (1-\lambda) L^\circ )^\circ) 
 \le
\vol(K) \, \vol(L),$$ or, equivalently,
\begin{equation*}
\vol(K \vee - L ) \, \vol (( K^\circ +  L^\circ )^\circ)  \le
\vol (K) \, \vol(L). \end{equation*}
The proof of Theorem \ref{thm:strange} is thus complete. 			
\end{proof}

\section{The planar case}\label{sec:R2}
While Godbersen's conjecture is clearly true for $n=2$ by the inclusion
$-K \subset 2K$ for centered convex regions in the plane,
the validity of the other two conjectures presented in Section
\ref{sec:intro} is not as self-evident. In this
section we assert the validity of Conjecture~\ref{Conj_Gen-Far-Red} when
$n=2$. We note that in the plane, the case $\lambda=\frac{1}{2}$ (along
with a characterization of the equality case) was first established by
Estermann \cite{Estermann}, and later by Levi \cite{Levi}, F\'{a}ry
\cite{Fary} and Yaglom and Boltyanski{\u\i} \cite{YaglomBolt}. We refer the
reader to \cite{Einhorn} for a detailed survey, which contains not only the
history of these problems and several proofs, but also similar related
problems dealing with intersections of the convex bodies $K$ and $-t K$.

We show that Conjecture \ref{Conj_Gen-Far-Red} holds in the
plane (for every $\lambda\in[0,1]$):

\begin{thm}\label{Thm_Planar-GFR}
Let $K\subseteq\RR^2$ be a convex body, and let $\lambda\in[0,1]$. If
$x$ is the center of mass of $K$, then $$
	{\rm Area}((1-\lambda)(K-x)\vee\lambda(x-K))\le
	{\rm Area}((1-\lambda)S\vee-\lambda S),$$
where $S$ is a centered triangle such that
${\rm Area}(K)={\rm Area}(S)$.
\end{thm}

\begin{proof}[{\bf Proof of Theorem \ref{Thm_Planar-GFR}}]
Without loss of generality, we may assume that $K$ is centered and has
unit area, and show that $${\rm Area}((1-\lambda)K\vee-\lambda K)\le
{\rm Area}((1-\lambda)S\vee-\lambda S).$$ We may further assume that $K$
is a polygon, and by a standard  continuity argument, the result will
hold for a general convex body.

We will describe an inductive process, in which we remove the vertices
of $K$ one by one until reaching a triangle. In each step, we replace
the body $K$ with a body $K'$ which has one vertex less, without
changing the area or the center of  mass, and  without decreasing the
area of $(1-\lambda)K\vee-\lambda K$. Let $K = {\rm Conv} \{x_1,\ldots,
x_N\}$, where $N \ge 4$. For $t\in\RR$, consider the body
$$\widetilde{K}_t= {\rm Conv} \{x_1,x_2 + tu,x_3,\ldots, x_N\},
\ {\rm with} \ u=x_3-x_1.$$ Denote by $l_1$ and $l_3$ the lines
containing the segments $[x_N,x_1]$ and $[x_3,x_4]$ respectively. Let
$\alpha < 0< \beta$ be such that $x_2 +\alpha u \in l_1$, and $x_2 +
\beta u \in l_3$. Note that for $t\in [\alpha, \beta]$ one has
$\widetilde{K}_t= {\rm Conv} \{x_1,x_3,\ldots, x_N\} \cup {\rm Conv}
\{x_1,x_2+tu,x_3\}$, and thus ${\rm Area}(\widetilde{K}_t) =
{\rm Area}(K)$. Moreover, the center of mass of $\widetilde{K}_t$ is
$\theta  t u$, where $\theta= \frac{1}{3} {\rm Area}({\rm Conv}
\{x_1,x_2,x_3\})$. Indeed, the center of mass of a union of two planar
sets is the average of the centers of mass, weighted  by their
respective areas. Denote $K_t = \widetilde{K}_t - \theta t u$. Note
that for $t \in [\alpha,\beta]$, $K_t$ has its center of mass at the
origin, and  ${\rm Area}(K_t) = {\rm Area}(K) = 1$.

		\begin{figure}[h]
			\centering
			\includegraphics[width=0.4\textwidth]{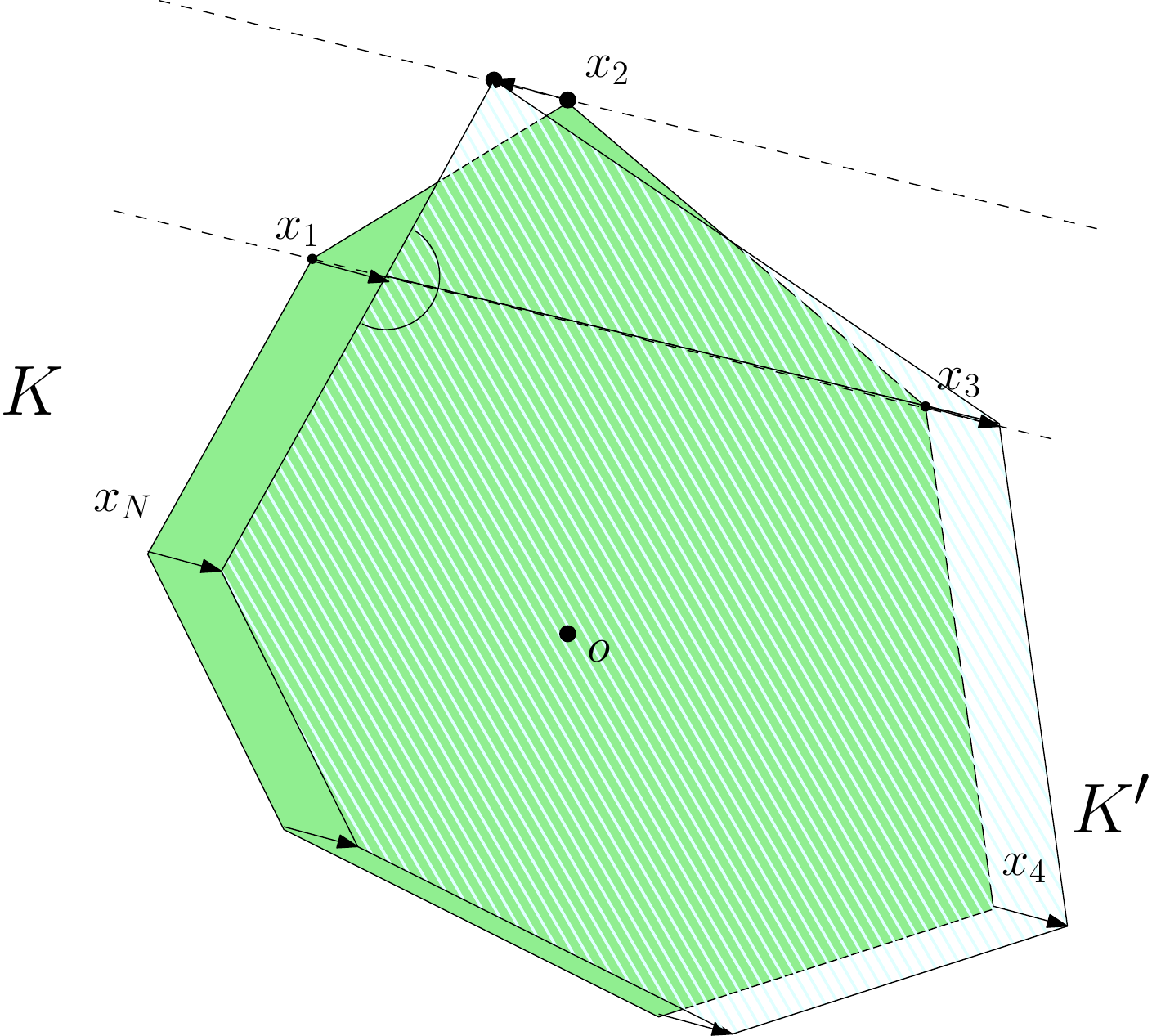}
			\caption{The area and center of mass are preserved, and $K'$ has one vertex less than $K$.}
			\label{fig:delete-vertex}
		\end{figure}

The family of bodies $\{(1-\lambda)K_t \vee-\lambda K_t\}_{t\in [\alpha, \beta]}$ is a linear parameter system (see \cite{RS3} for the definition), since all vertices are moving parallel to $u$. Therefore the area of $(1-\lambda)K_t \vee-\lambda K_t$ is a convex function of $t$, and attains its maximum in one of the edges of the interval $[\alpha,\beta]$. Assume the maximum is achieved at $\alpha$, and set $K'=K_{\alpha}$. Then, in addition to having unit area and center of mass at the origin, $K'$ satisfies $$
{\rm Area}((1-\lambda)K \vee-\lambda K ) \le {\rm Area}((1-\lambda)K'\vee-\lambda K' ),$$
as $K=K_0$. Since $x_2 +\alpha u \in l_1$, $K'$ has one less vertex than $K$, and thus the proof is complete.
\end{proof}

\section{Appendix}
We give here another proof of Theorem~\ref{thm:K-L-vol}. 
The proof is attained by essentially repeating the arguments from~\cite{RS2}, but instead of considering $K$ and $-K$, we consider two general convex bodies $K$ and $L$. Moreover, we will be able to characterize the equality case in Theorem~\ref{thm:ab-vol}.

\subsection{The Rogers--Shephard body}
We consider the following $(2n+1)$-dimensional body (see Figure 2), a
special case of which, where $K=L$, plays a central role in~\cite{RS2}.
\[
G{(K,L)}:=\left\{ \left(x,y,\theta\right)\in\RR^{n}\times\RR^{n}\times\RR\,|\,\theta\in\left[0,1\right],\, x\in\theta K,\, x+y\in (1-\theta) L\right\} .
\]

\begin{figure}[h!] \label{fig-RS-Body}
	\centering
	\includegraphics[width=0.5\textwidth]{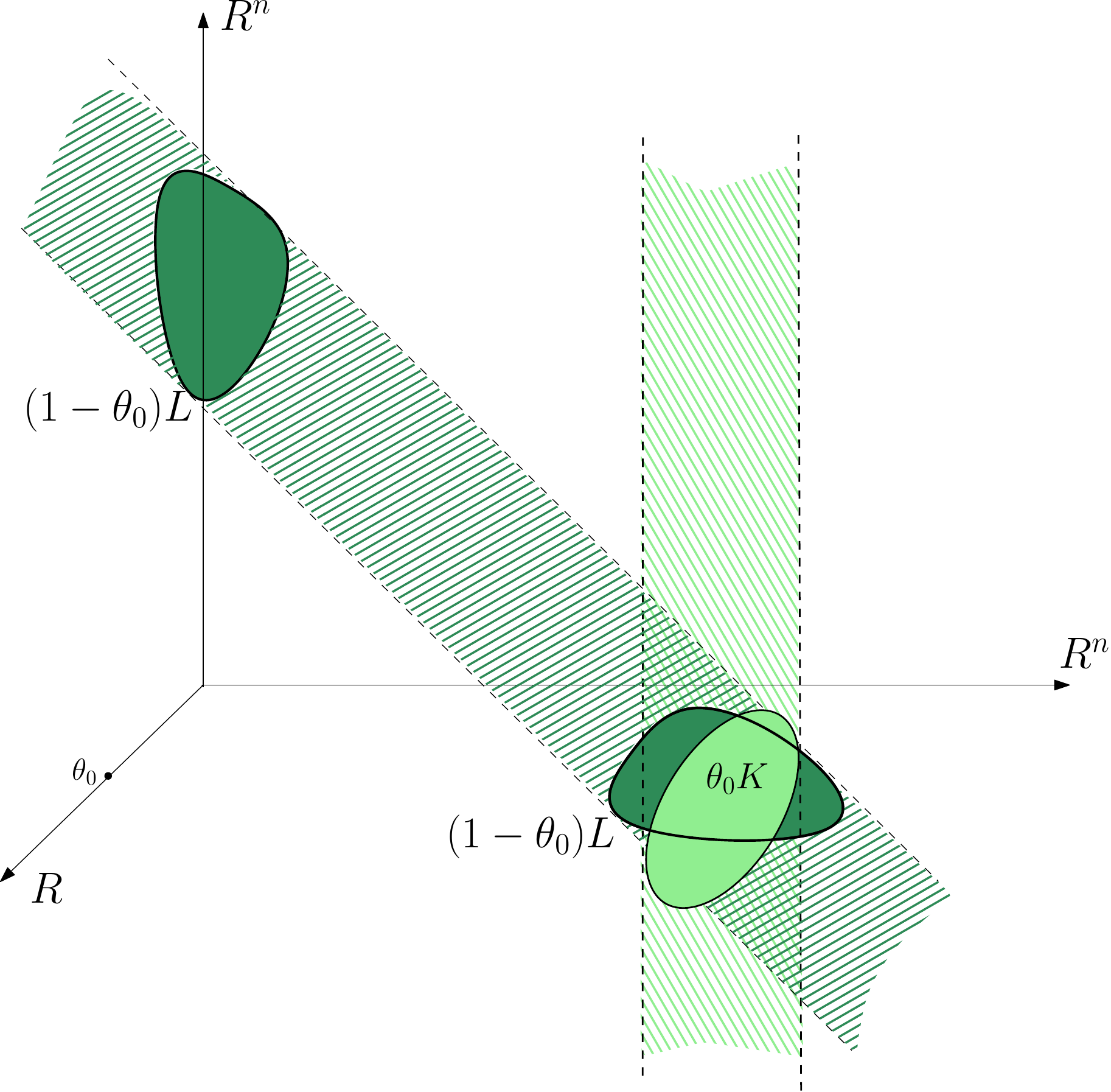}
	\caption{The section of $G{(K,L)}$ where $\theta=\theta_0,\,y=0$, is an intersection of two cylinders.}
\end{figure}

The projection of $G{(K,L)}$ onto the $\left(n+1\right)$-dimensional subspace of points of the form $\left(0,y,\theta\right)$ is denoted by $C(K,L) =\{ (0,y,\theta) \ | \ \theta \in [0,1], \ y \in (1-\theta) L - \theta K \}$. Equivalently, 
\begin{equation}\label{eq:C-a,b}
C(K,L)=\{0\}\times {\rm conv}\left(L\times\left\{ 0\right\} ,\,-K\times\left\{ 1\right\} \right)\subseteq\{0\}\times \RR^{n+1}.
\end{equation}

\noindent When $K=L$, this is exactly the body $C(K)$ used in \cite{RS2}. 

The main tool used by Rogers and Shephard in \cite{RS2} for finding an
upper bound for the volume of the difference body is the following
theorem which we  provide along with its proof, for the convenience of
the reader.

\begin{thm}[{\bf Rogers and Shephard}] \label{lem:volBoundBySecProj}
Let $K\subset\RR^n$ be a convex body, and let $H=K\cap E$ be a
$j$-dimensional section of $K$, and  $L$ the orthogonal
projection of $K$ onto $E^\perp$. Then
	\begin{equation}\label{eq:RSlemm}
		\frac{j!\left(n-j\right)!}{n!}\vol_{j}\left(H\right)\vol_{n-j}\left(L\right)\le \vol_{n}\left(K\right).
	\end{equation}
Moreover, equality holds if and only if for every direction $v\in E^\perp$, the intersection of
$K$ with $E+\R^+v$ is obtained by taking the convex hull of $H$ and one more point.
\end{thm}

\begin{proof}[{\bf Proof of Theorem \ref{lem:volBoundBySecProj}}]
The proof consists of two arguments. The first one is that all of the quantities in \eqref{eq:RSlemm}
are invariant under a Schwarz symmetrization, so we may assume that all intersections of $K$ parallel to $H$ are (centered) dilates of a ball. That is, we
may consider the body:
$$ K^*=\left\{(l,y)\in\R^{n-j}\times\R^j\quad |\quad l\in L, \ |y|\le \left(\frac{\vol_j(K_l)}
{\vol_j(B_2^j)}	\right)^{1/j}  \right\},$$ 
where $K_l=K\cap (E+l)$, for every $l\in L$.
Their second argument is that $H^*\vee L \subseteq K^*$, where 
$H^*$ denotes the section $K^* \cap E$.  Then, a simple
computation shows that the left-hand side of \eqref{eq:RSlemm} equals
$\vol_n(H^*\vee L)$, and thus inequality \eqref{eq:RSlemm} holds.

Assume now that equality holds in  \eqref{eq:RSlemm}. First note that 
$H^*\vee L\subseteq K^*$,
thus equality in volumes implies $K^*=H^*\vee L=H^*\vee \partial L$.
Moreover, for every direction $v\in E^{\perp}$, let $[0,l] = L\cap (\R^+v)$, and for
every $t\in[0,1]$, let $f_v(t)=\vol_{j} (K_{tl})^{1/j}$. Note that $f_v$ is invariant
under Schwartz symmetrization, hence for every $ v\in E^{\perp}$, $f_v$ is linear, and $f_v(1)=0$.
Since $K_{tl}$ contains  $t K_{l}+(1-t)H$, from Brunn Minkowski's inequality one has
$$ f_v(t)=\vol_{j}(K_{tl})^{1/j} \geq t \vol_{j}(K_{l})^{1/j} + (1-t) \vol_{j}(H)^{1/j} = (1-t)f_v(0).$$
Thus, from the equality case in Brunn Minkowski's inequality, $K_l$
is a homothety of $H$ of zero volume, i.e. it is a point. Moreover, for every $t \in [0,1]$, one has 
$K_{tl} = t K_{l}+(1-t)H$, and thus $K \cap (E+\R^+v) = H \vee K_l$, and the proof is complete. 
\end{proof}

\subsection{An upper bound for the volume of $C(K,L)$}

\begin{thm}\label{thm:RS-K-L-bound}
		For convex bodies $K,L\subseteq\RR^n$, let $C(K,L)$ be as defined in \eqref{eq:C-a,b}. For every $\theta\in[0,1]$,
		\[
		\vol_{n+1}\left(C\left(K,L\right)\right)\leq\frac{1}{n+1} \left ( \frac{\vol_{n}\left(K\right)\vol_{n}\left(L\right)}{\vol_n(\theta K\cap(1-\theta) L) } \right ).
		\]
\end{thm}

\noindent In \cite{RS2} this bound was obtained for the specific case $K=L$ and $\theta=\frac{1}{2}$.

\begin{proof}[{\bf Proof of Theorem \ref{thm:RS-K-L-bound}}]
	In order to estimate ${\rm Vol}_{n+1}\left(C\left(K,L\right)\right)$, we apply Theorem \ref{lem:volBoundBySecProj} for the body $G{(K,L)}$, and the  $n$-dimensional affine subspace  $E= \{ \theta = \theta_0, y=0 \}$.
	First, by Fubini's Theorem,
	\begin{eqnarray*}
		\vol_{2n+1}\left(G{(K,L)}\right) & = & \int_{0}^{1}d\theta\int_{\theta K}\vol_n  ((1-\theta)L )dx\\
		& = & 
		\vol_n\left(K\right)\vol_n\left(L\right)\int_{0}^{1}\theta^n\left(1-\theta\right)^n d\theta\\
		& = & 
		\vol_n\left(K\right)\vol_n\left(L\right)\frac{n!n!}{\left(2n+1\right)!}.
	\end{eqnarray*}
As in Theorem \ref{lem:volBoundBySecProj}, set $H = G{(K,L)} \cap E$.
Note that $\vol_n(H) =\vol_n(\theta_0 K\cap (1-\theta_0)L)$. As
mentioned before, the projection of $G{(K,L)}$ on $E^{\perp}$ is exactly
$C\left(K,L\right)$, and using Theorem  \ref{lem:volBoundBySecProj} we get
	\begin{equation}\label{eq:sections-of-G(K,L)}
		\frac{n!\left(n+1\right)!}{\left(2n+1\right)!}
		\vol_{n+1}\left(C\left(K,L\right)\right)
		\vol_{n}\left(\theta_0 K\cap (1-\theta_0)L \right) \leq
		\vol_{2n+1}\left(G{(K,L)}\right).
	\end{equation}
Plugging in the volume of $G{(K,L)}$ completes the proof of the theorem. 
\end{proof}

\subsection{A second proof of Theorem \ref{thm:K-L-vol}}
Using Theorem \ref{lem:volBoundBySecProj}, this time for the body
$C(K,L)$, and the volume bound from Theorem \ref{thm:RS-K-L-bound},
we can give yet another proof of Theorem \ref{thm:K-L-vol}.

\begin{proof}[{\bf Proof of Theorem \ref{thm:K-L-vol}}]
Let $K,L\subseteq \RR^n$ be two convex bodies such that $0\in K\cap L$,
and set $\theta\in[0,1]$. We need to show that	\[
\vol(L \vee -K) \, \vol( \theta K\cap (1-\theta) L) \le
\vol(K)\vol(L).    \]
Let $E$ be the 1-dimensional subspace of ${\mathbb R}^{n+1}$ given by
$E = \{ x =0 \}$. The body $L \vee -K$ is the $n$-dimensional projection
of $C(K,L)$ onto the subspace $E^{\perp}= \{(x,0) \, | \, x\in\RR^n\}$.
Since $0 \in K \cap L$, the section $H = E \cap C(K,L)$ is a unit segment.
By Theorem \ref{lem:volBoundBySecProj},
	\begin{equation}\label{eq:C-versus-(K-L)}
	 \frac{1}{n+1}\vol_{n}(-K\vee L) \leq \vol_{n+1}(C(K,L)).
	\end{equation}
Combining this with the volume bound for $C(K,L)$ established in
Theorem \ref{thm:RS-K-L-bound} above, we get the desired inequality.
\end{proof}

\subsection{The equality case in Theorem \ref{thm:ab-vol} and in Theorem \ref{thm:K-L-vol}}

Here we characterize the equality cases in Theorems  \ref{thm:ab-vol}
and \ref{thm:K-L-vol}. We start with the former, and show that equality
holds in~$(\ref{ineq-needed-for-God})$ if and only if $K$ is a simplex
with a vertex at the origin. Indeed, it is not hard to check that for
the standard simplex $S$ one has \[
\vol((1-\lambda)S\vee-\lambda S)=
\sum_{k=0}^n \binom{n}{k}(1-\lambda)^k \lambda^{n-k}\vol(S)
=\vol(S).\]
As for the other direction, assume that $\vol((1-\lambda)K\vee -\lambda K)=
\vol(K)$. We wish to show that $K$ is a simplex with $0$ as one of its
vertices. Recall from the introduction that Theorem \ref{thm:K-L-vol},
for $(1-\lambda)K, \lambda K$, and $\theta_0 = \lambda$, immediately
yields inequality~$(\ref{ineq-needed-for-God})$. Combining
\eqref{eq:sections-of-G(K,L)} and \eqref{eq:C-versus-(K-L)} in this case yields
\begin{equation}\label{eq:C-versus-(aK-V-(x-bK))}
	\vol_n(-(1-\lambda)K\vee \lambda K)\le 
	(n+1) \vol_{n+1}(C((1-\lambda)K,\lambda K))\le
	\vol_n(K).
\end{equation}
From the assumption $\vol((1-\lambda)K\vee -\lambda K)=\vol(K)$ it
follows that both inequalities in~$(\ref{eq:C-versus-(aK-V-(x-bK))})$
are in fact equalities. In particular, equality holds in
\eqref{eq:sections-of-G(K,L)} for the bodies $(1-\lambda)K, \lambda K$,
and $\theta_0 = \lambda$. By the equality condition in Theorem
\ref{lem:volBoundBySecProj}, this implies in particular that sections
of the body $G{((1-\lambda)K,\lambda K)}$ by affine subspaces of the
form $\{(x,y,\theta)| x\in \R^n, y=y_0, \theta =\lambda\}$, for any
$y_0\in\lambda(1-\lambda)K-\lambda(1-\lambda)K$ (which are given by
$\lambda(1-\lambda)K \cap \lambda(1-\lambda)K -y_0$), are homothetic.
Thus, $K$ must be a simplex by the following lemma due to Rogers and
Shephard.

\noindent {\bf Lemma 4 from \cite{RS}}. {\it Let $K$ be a convex body
in $\RR^n$. If  the intersections $K \cap (K+x)$ are homothetic for all
$x \in K - K$, then $K$ is a simplex.  }

\noindent Finally, in order to show that $0$ is a vertex of $K$, note
that equality holds also in~$(\ref{eq:C-versus-(K-L)})$, for the bodies
$(1-\lambda)K, \lambda K$. Hence, among all sections of $C((1-\lambda)K,
\lambda K)$ parallel to $H = E \cap C((1-\lambda)K, \lambda K)$, $H$ is
the only one with unit length. Since we assumed that $K$ contains the
origin, one has $(1-\lambda) K \cap - \lambda K=\{0\}$. This, together
with the fact that $K$ is a simplex, means that $0$ must be one of the
vertices of $K$. 

We turn now to show that equality in Theorem  \ref{thm:K-L-vol} holds
if and only if $K$ and $L$ are simplices with a common vertex at the
origin, and such that $(1-\theta)L = \theta K$. To this end, we shall
make use of Theorem \ref{thm:strange}. Assume that for some given
$\theta\in (0,1)$,
\begin{equation} \label{eq-in-1.6}
\vol(L \vee -K) \, \vol(\theta K\cap(1-\theta) L) =  \vol(K)\vol(L).
\end{equation}
Since the inclusion $\theta K\cap(1-\theta) L\subseteq \left( K^\circ +L^\circ\right)^\circ$ always holds, by Theorem \ref{thm:strange} we then must have 
 \[ \theta K\cap(1-\theta) L = \left( K^\circ +L^\circ\right)^\circ \]
(as both are compact convex sets, inclusion together with equality of volumes implies equality of sets). 
We claim that this equality implies that $K$ and $L$ are homothetic.
Indeed, we may rewrite the above equality as
 \[ \theta^{-1} K^\circ \vee (1-\theta)^{-1} L^\circ = K^\circ +L^\circ, \]
so that in particular 
\[ \theta^{-1} K^\circ  \subset K^\circ +L^\circ, \qquad (1-\theta)^{-1} L^\circ  \subset K^\circ +L^\circ.\]
Thus, $\theta^{-1}h_{K^\circ}\le  h_{K^\circ}+h_{L^\circ}$, and
$(1-\theta)^{-1}h_{L^\circ}\le  h_{K^\circ}+h_{L^\circ}$. Putting the two together one has
\[h_{L^\circ} = \frac{1-\theta}{\theta}h_{K^\circ}, \]
or equivalently, $(1-\theta)L = \theta K$. This, together
with~$(\ref{eq-in-1.6})$ implies that equality in~$(\ref{ineq-needed-for-God})$
holds for $K$. Thus, from the characterization of the equality case in Theorem
\ref{thm:ab-vol}, the body $K$ must be a simplex with a vertex at the origin.
This completes the proof for the equality case in Theorem \ref{thm:K-L-vol}. 

\begin{rem} {\rm
It is worthwhile to notice that in Theorem \ref{thm:strange} there are more equality cases than in Theorem \ref{thm:ab-vol}. Indeed, one may readily check that for positive $\lambda_1,\dots,\lambda_n$, and the bodies $K=\conv \{ 0, e_1, \ldots, e_n\}$, $L = \conv \{ 0, \lambda_1e_1, \ldots, \lambda_n\} $ we have that $\vol(K)\vol(L) = \frac{1}{n!^2}\prod_{i=1}^n \lambda_i$, that 
\[ \vol(\left( K^\circ +L^\circ\right)^\circ)= \frac{1}{n!}\prod_{i=1}^n\frac{\lambda_i}{1+\lambda_i}\]
and that 
\[ \vol(K \vee -L)= \frac{1}{n!}\sum_{A\subset\{1,\ldots n\} }\prod_{i\in A} \lambda_i  = \frac{1}{n!}\prod_{i=1}^n (1+ \lambda_i). \] 
The characterization of the equality case in Theorem \ref{thm:strange} is thus left open. }
\end{rem}

\bibliographystyle{amsplain}
\addcontentsline{toc}{section}{\refname}\bibliography{Two-remarks-on-God}

\noindent Shiri Artstein-Avidan, Dan Florentin, Keshet Einhorn, Yaron Ostrover\\
\noindent School of Mathematical Science, Tel Aviv University, Tel Aviv, Israel.\\

\noindent shiri@post.tau.ac.il  \\
 danflore@post.tau.ac.il  \\ keshetgu@mail.tau.ac.il \\
  ostrover@post.tau.ac.il

\end{document}